\documentclass[10pt]{article}
\title{\vskip-2.0em A note on some group $\Cst$-algebras which are quasi-directly finite}
\author{Yemon Choi}

\RequirePackage{amsmath,amssymb,amsfonts}

\usepackage[usenames]{color}
\newcommand{\dt}[1]{\textcolor{Maroon}{\textsf{#1}}}

\newcommand{\defeq}{:=}
\renewcommand{\emph}[1]{{\sl #1}\/} 

\newenvironment{YCnum}{%
\begin{enumerate}

}{\end{enumerate}\ignorespacesafterend}

\newcommand{\Real}{{\mathbb R}}
\newcommand{\Cplx}{\mathbb C}

\newcommand{\Bdd}{{\mathcal B}}
\newcommand{\Cpct}{{\mathcal K}}

\newcommand{\fu}[1]{{#1}^{\sharp}}    
\newcommand{\cu}[1]{{#1}_{\rm un}}  
\newcommand{\id}[1][]{{\sf 1}_{#1}} 

\newcommand{\abs}[1]{\vert{#1}\vert}

\newcommand{\norm}[1]{\Vert{#1}\Vert}

\newcommand{\VN}{\operatorname{VN}}
\newcommand{\FA}{\operatorname{A}}
\newcommand{\Cst}{\ensuremath{\operatorname{C}^*}}
\newcommand{\SIN}{\ensuremath{\operatorname{SIN}}}
\newcommand{\SL}{\operatorname{SL}}

\newcommand{\ran}{\operatorname{ran}}
\newcommand{\qm}{\mathrel{\bullet}} 

\newcommand{\cA}{{\mathcal A}}
\newcommand{\cF}{{\mathcal F}}
\newcommand{\cM}{{\mathcal M}}
\newcommand{\cT}{{\mathcal T}}
\newcommand{\fg}{{\mathfrak g}}
\newcommand{\fm}{{\mathfrak m}}
\newcommand{\fn}{{\mathfrak n}}

\newcommand{\sR}{{\sf R}}

\newcommand{\al}{\alpha}
\newcommand{\lm}{\lambda}
\newcommand{\Gm}{\Gamma}

\newcommand{\pair}[2]{\langle#1,#2\rangle}

\RequirePackage{amsthm}   

\newcounter{pulse}[section]
\numberwithin{pulse}{section}  

\theoremstyle{plain}
\newtheorem{thm}[pulse]{\sc Theorem}
\newtheorem{prop}[pulse]{\sc Proposition}
\newtheorem{lem}[pulse]{\sc Lemma}
\newtheorem{cor}[pulse]{\sc Corollary}
\theoremstyle{definition}
\newtheorem{dfn}[pulse]{\sc Definition}

\theoremstyle{remark}
\newtheorem{rem}[pulse]{\sc Remark}

\begin{document}
\maketitle

\begin{abstract}
An algebra is said to be quasi-directly finite when any left-invertible element in its unitization is automatically right-invertible. It is an old observation of Kaplansky that the von Neumann algebra of a discrete group has this property; in this note, we collate some analogous results for the group $\Cst$-algebras of more general locally compact groups. Partial motivation comes from earlier work of the author on the phenomenon of empty residual spectrum for convolution operators.

\bigskip\noindent {MSC 2010: Primary 22D25; Secondary 46L05}
\end{abstract}

\begin{section}{Introduction}
Following Munn~\cite{Munn_DF1}, we say that an algebra $R$ is \dt{quasi-directly finite} if its unitization $\fu{R}$ is \dt{directly finite}; that is, if every left invertible element in $\fu{R}$ is automatically right invertible. In the present article, we consider the question of when the reduced or full $\Cst$-algebras of a locally compact group are quasi-directly finite.
Apart from intrinsic interest, this question is motivated by arguments in the author's previous article \cite{YC_surjunc}, where attention was mostly confined to the case of discrete groups. In this case, the group von Neumann algebras are known to be directly finite; one can exploit this to show that, in a variety of settings, convolution operators associated to actions of discrete groups have empty residual spectrum.

As a hint that the situation will be subtler than in the discrete case, consider the reduced group $\Cst$-algebra and the group von Neumann algebra of $\SL(2,\Real)$, denoted by $\Cst_r(\SL(2,\Real))$ and $\VN(\SL(2,\Real))$ respectively. Both algebras have been studied in some detail, and it follows from standard facts about them, that while $\VN(\SL(2,\Real))$ is \emph{not} directly finite, the unitization of $\Cst_r(\SL(2,\Real))$~\emph{is}.
More generally, we shall see below (Theorem~\ref{t:unimodular-qdf}) that $\Cst_r(G)$ is quasi-directly finite if $G$ is \emph{unimodular}.

There seems to be no known characterization, in terms of $G$, of when $\Cst_r(G)$ is quasi-directly finite. One of our secondary aims is to collate some of the relevant partial results in one place for ease of reference, as a precursor to possible further work.

\subsection*{Residual spectrum and directly finite algebras}
In order to give some background and motivation to what follows, we shall outline the link between the residual spectrum of an operator and direct finiteness of an algebra which contains it. Recall that if $T$ is a bounded linear operator on a Banach space $X$, then the \dt{residual spectrum} of $X$ is the set of all $\lm\in\Cplx$ such that $T-\lambda I$ is injective with closed, proper range; put slightly differently, it is the set of all $\lm$ in the spectrum of $T$ which are not approximate or genuine eigenvalues. Thus, knowing in advance that the residual spectrum of $T$ is empty is of interest in working out the spectral theory of~$T$.

Let us say that a subalgebra $A\subseteq\Bdd(X)$ is \dt{surjunctive} if every $a\in A$ has empty residual spectrum (note that this is, {\it a priori}\/, not just a property of the algebra $A$ as an abstract algebra, but a property of $A$ and its realization inside $\Bdd(X)$). If~$A$ is surjunctive, then a short argument shows that the subalgebra of $\Bdd(X)$ generated by $A$ and the identity operator must be directly finite; the converse is in general false, as illustrated by a construction of G.~A. Willis (see \cite[Proposition~4.1]{YC_surjunc} and the surrounding remarks for some discussion of this).

On the other hand, under extra conditions, we do get a partial converse:

\begin{thm}\label{t:surjunc}
Let $A\subseteq\Bdd(X)$ be a \emph{closed} subalgebra. Suppose that $A$ is bicontinuously isomorphic to a quasi-directly finite $\Cst$-algebra. Then $A$ is surjunctive.
\end{thm}

\begin{proof}[Proof sketch]
This is essentially proved in \cite{YC_surjunc}, but the arguments are somewhat scattered and are moreover restricted to the case where $A$ is unital. Thus, for sake of convenience, we sketch how the argument goes.

Let $a\in A$, $\lm\in\Cplx$, and suppose $a-\lm I$ is injective but not invertible. It suffices to find a sequence $(b_n)$ in $A$ such that $\norm{(a-\lm I)b_n}\to 0$ while $\inf_n \norm{b_n} > 0$; for, given such a sequence, it follows that there is a sequence $(x_n)$ in $X$ of approximate eigenvectors for $a$ corresponding to $\lm$, showing that $\lm$ is not in the residual spectrum.

It will follow from Lemma~\ref{l:unitizations} below that since $A$ is quasi-directly finite, the sub\-algebra $\cu{A}\subseteq\Bdd(X)$ that is generated by $A$ and $I$ will be directly finite; moreover, $\cu{A}$ is bicontinuously isomorphic to a \emph{unital}, directly finite $\Cst$-algebra. Note also that left multiplication by $a-\lm I$ must be an injective operator from $A$ to itself. The existence of a sequence $(b_n)$ with the required properties now follows exactly as in the proof of \cite[Theorem~3.2]{YC_surjunc}.
\end{proof}

One way to obtain examples satisfying the conditions of Theorem~\ref{t:surjunc} is as follows. Let $\cM$ be a semifinite von Neumann algebra, equipped with a faithful, normal, semifinite trace $\tau$ on $\cM$, let $1\leq p\leq\infty$, and let $L^p(\cM,\tau)$ be the noncommutative $L^p$-space associated to $(\cM,\tau)$, as constructed in e.g.~\cite{Dix_BSMF53}; then $L^p(\cM,\tau)$ has the structure of an $\cM$-bimodule in a natural way, and the corresponding homomorphism $\imath_\tau: \cM \to \Bdd(L^p(\cM,\tau))$ is injective with closed range. Thus, if $\cA$ is a quasi-directly finite $\Cst$-subalgebra of $\cM$, Theorem~\ref{t:surjunc} implies that $\imath_\tau(\cA)$ will be a surjunctive subalgebra of $\Bdd(L^p(\cM,\tau))$.

By a result of Dixmier~\cite{Dix_semifin}, $\VN(G)$ is semifinite for every connected locally compact group $G$, and so the discussion of the previous example can be applied.
It is worth noting that for \emph{any} locally compact group~$G$, there is a natural action of $\VN(G)$ on the \dt{Fourier algebra} $\FA(G)$ (see \cite{Eym_BSMF64} for the definition). It is not hard to show that the corresponding homomorphism $\imath_{\FA}: \VN(G) \to \Bdd(\FA(G))$ is an isometry, giving us another possible source of examples.

\begin{rem}
Given a von Neumann algebra $\cM$, its predual $\cM_*$ has a natural $\cM$-bimodule structure. Moreover:
\begin{YCnum}
\item if $\cM$ has a faithful normal semifinite trace $\tau$, then it can be shown that there is an isomorphism of $\cM$-bimodules from $L^1(\cM,\tau)$ onto $\cM_*$;
\item if $\cM=\VN(G)$ for some locally compact group $G$, then there is an isomorphism of $\cM$-bimodules from $\FA(G)$ onto $\VN(G)_*$.
\end{YCnum}
Thus if $\VN(G)$ is semifinite, with a trace $\tau$, this gives us two ways to view the action of $\VN(G)$ on its predual.
\end{rem}

\end{section}

\begin{section}{Notation and preliminaries}
It is convenient, in phrasing some of our results, to use the notions of left and right quasi-inverses.

\begin{dfn}
Given $a,b\in A$, let $a\qm b\defeq a+b-ab$. If $a\qm b=0$ then we say that $a$ is a \dt{left quasi-inverse} for $b$ and $b$ is a \dt{right quasi-inverse} for $a$\/.
\end{dfn}

The idea, of course, is that if $A$ has an identity element $\id$, then $\id-a\qm b = (\id-a)(\id-b)$. It is more intuitive to reason with left, right and two-sided invertible elements than with their ``quasi-'' counterparts; but since we will be working with rings that may or may not have identity elements, the language of quasi-inverses streamlines some of the statements. This is particularly true when we start to move between various ideals in non-unital rings, where adjoining an identity would destroy the ideal property and make the phrasing of various conditions slightly cumbersome.

\subsection*{Basic properties}
These are surely well-known, but we collect them here for ease of reference. Note first of all that $\qm$ is associative: although one could check this by a direct comparison of $(a\qm b)\qm c$ with $a\qm(b\qm c)$, it is more instructive to adjoin a formal identity $\id$ and observe that
\[ \id-(a\qm b)\qm c = (\id-a\qm b)(\id-c) = (\id-a)(\id-b)(\id-c) = (\id-a)(\id-b\qm c) = \id-a\qm (b\qm c) \] 

An easy yet fundamental fact about unital Banach algebras is that the group of invertible elements is open in the norm topology. This has an obvious analogue for quasi-inverses; we state a slightly more precise version below, for later reference.

\begin{lem}\label{l:stable}
Let $A$ be a Banach algebra and let $a\in A$. Suppose $a$ has a right quasi-inverse in $A$; then so does every $a'$ that is sufficiently close to $a$.
\end{lem}

\begin{proof}
Suppose there exists $b\in A$ such that $a\qm b=0$. Adjoining an identity element $\id$ to $A$, chosen to have norm $1$, we thus have $(\id-a)(\id-b)=\id$.  
Put $\delta=(1+\norm{b})^{-1}>0$. Then, given any $a'\in A$ such that ${a'-a} < \delta$\/, put
\[ u \defeq (\id-a')(\id-b) \in \fu{A} \]

We have $\id-u = (a'-a)(\id-b)$ which has norm $< 1$; thus $u$ is invertible in the Banach algebra $\fu{A}$. Moreover, the usual formula for the inverse shows that
\[ c\defeq \id-u^{-1} = - \sum_{n\geq 1} (\id-u)^n \]
lies in $A$. Since $\id-u=a'\qm b$ by construction, we see that $(a'\qm b)\qm c=0$. By associativity of $\qm$, it follows that $b\qm c$ is a right quasi-inverse to $a'$.
\end{proof}

We will also use the following simple observation, whose proof is straightforward from the definitions.
\begin{lem}\label{l:ideal}
Let $A$ be a $k$-algebra and $J$ a right ideal in $A$. If $a\in J$ has a right quasi-inverse $b\in A$, then $b\in J$.
\end{lem}

\subsection*{Quasi-directly finite algebras}
A ring $\sR$ with identity $\id$ is said to be \dt{directly finite} if any $x,y\in\sR$ which satisfy $xy=\id$ necessarily satisfy $yx=\id$. Many of the examples considered in the present article are algebras without an identity element, and while we can always pass to the unitization, it is more convenient to be able to work within the original algebra: see Proposition~\ref{p:qdf-ideal} below for an example of this. Thus, following Munn~\cite{Munn_DF1}, we make the following definition.

\begin{dfn}
Let $\sR$ be a ring, not necessarily having an identity element. We say $\sR$ is \dt{quasi-directly finite} if every element which is left quasi-invertible is also right quasi-invertible. (By our earlier remarks, if an element has both a left quasi-inverse $b_L$ and a right quasi-inverse $b_R$, then $b_L=b_R$\/.)
An algebra is said to be quasi-directly finite if its underlying ring~is.
\end{dfn}

For sake of brevity, we shall henceforth abbreviate ``quasi-directly finite'' to ``q.d.f''.

\begin{rem}
We gave our definition in terms of quasi-inverses, since this is the formalism we will use in following sections. One can rephrase the definition as follows: $\sR$ is q.d.f.\ if and only if, for every $a,b\in\sR$ satisfying $a+b=ab$, we have $a+b=ba$\/.
\end{rem}

\begin{lem}\label{l:unitizations}
Let $k$ be a field.
\begin{YCnum}
\item\label{item:FU} If $A$ is a q.d.f. $k$-algebra, then its forced unitization $\fu{A}$ is directly finite.
\item\label{item:qdf-df} If $B$ is a $k$-algebra with an identity element, then $B$ is q.d.f.\ if and only if it is directly finite.
\end{YCnum}
\end{lem}
\begin{proof}[Proof sketch]
If $A$ is q.d.f. and $\fu{a},\fu{b}\in\fu{A}$ satisfy $\fu{a}\fu{b}=\id$, then we must have $\fu{a}=\lm(\id-a)$ and $\fu{b}=\lm^{-1}(\id-b)$ for some $a,b\in A$ and $\lm\in\Cplx$; by rescaling if necessary, we may assume without loss of generality that $\lm=1$. But then we have $a\qm b=\id-\fu{a}\fu{b}=0$; since $A$ is q.d.f.\ this implies that $0=b\qm a =\id-\fu{b}\fu{a}$, so that $\fu{b}\fu{a}=\id$. This completes the proof of part~\ref{item:FU}. The proof of \ref{item:qdf-df} is similar, and we omit the details.
\end{proof}

\begin{rem}
Since a directly finite algebra is q.d.f, and since subalgebras of q.d.f. algebras are q.d.f, (ii) implies that we can reverse the implication in~(i).
\end{rem}

\begin{prop}\label{p:qdf-ideal}
Let $A$ be a Banach algebra and let $J$ be a \emph{right} ideal in $A$ which is dense for the norm topology. Then $J$ is q.d.f.\ if and only if $A$ is.
\end{prop}

\begin{proof}
Sufficiency is obvious, so we need only prove necessity.

Suppose $J$ is q.d.f. Let $a,b\in A$ satisfy $a\qm b=0$. By Lemma~\ref{l:stable} and density of $J$ in $A$, we can find $a'\in J$ which is close to $a$ and which has a right quasi-inverse in $A$, say $b'$; by Lemma~\ref{l:ideal}, $b' \in J$. Therefore, since $J$ is assumed to be q.d.f, $b'\qm a'=0$\/.

Moreover, the proof of Lemma~\ref{l:ideal} shows that we can take $b'$ to be of the form $b\qm c$ for some $c\in A$. Thus $b\qm (c\qm a')=0$, i.e. $b$ has a right quasi-inverse in $A$. Since we initially assumed that $a\qm b=0$, it follows that $b\qm a=0$ as required.
\end{proof}


\subsection*{A sufficient criterion for a $\Cst$-algebra to be q.d.f.}
Let $H$ be an infinite-dimensional Hilbert space. While the algebra $\Bdd(H)$ is evidently not directly finite, the closed subalgebra $\Cpct(H)$ is q.d.f. 
We shall see in this section that this is a special case of a more general result for semifinite von Neumann algebras.

\begin{prop}\label{p:qdf-cstar}
Let $A$ be a $\Cst$-algebra and $J$ a dense right ideal in $A$. Suppose that there exists a linear functional $\tau:J \to \Cplx$ with the following properties: (i) $\tau(ab)=\tau(ba)$ for all $a,b\in J$; and (ii) if $c\in J$ and $\tau(c^*c)=0$ then $c=0$\/. Then $A$ is q.d.f. 
\end{prop}

In the case where $A$ has an identity element and $J=A$, this result is well-known (and is indeed the basis of Montgomery's proof in \cite{Mon_dirfin} that the complex group algebra of a group is always directly finite). However, since I am unaware of a convenient reference where this case is stated explicitly, and since we are aiming for something slightly more general, we shall give a complete proof. The key observation is the following standard result about $\Cst$-algebras:

\paragraph{Fact.}
If $p$ is an idempotent in a unital $\Cst$-algebra, there exists a hermitian idempotent $e$ in that algebra which satisfies $ep=p$ and $pe=e$.

\medskip This result
seems to be part of the $\Cst$-algebraic folklore. In particular, in several sources it is merely observed, without further proof, that we can take $e$ to be 
\begin{equation}\label{eq:rabbit}
\tag{$\dagger$} e= pp^*(\id+(p-p^*)(p^*-p))^{-1}\,.
\end{equation}
This formula makes it clear that $e=fe$, but it is not so transparent that $ep=p$, nor that $e$ is an idempotent. Probably the easiest, if not the quickest, way to verify these properties is to regard $p$ as a projection inside $B(H)$, and hence as a $2\times 2$ operator matrix
$\left(\begin{matrix} I & R \\ 0 & 0 \end{matrix} \right)$
with regard to the decomposition of $H$ as $\ran(p)\oplus \ran(p)^\perp$\/. One now checks that the formula on the right hand side of \eqref{eq:rabbit} comes out to equal
$\left(\begin{matrix} I & 0 \\ 0 & 0 \end{matrix} \right)$,
i.e.~the orthogonal projection of $H$ onto $\ran(p)$; it is then clear that $ep=p$ and $pe=e=e^2$ as claimed.

\begin{proof}[Proof of Proposition~\ref{p:qdf-cstar}]
By Proposition~\ref{p:qdf-ideal}, it suffices to show that $J$ is q.d.f. Let $a,b\in J$ be such that $a\qm b=0$\/. Put $p = b\qm a = ab-ba\in J$\/; clearly $\tau(p)=0$, by the `tracial' property of $\tau$.
Moreover, since
\[ 2p-p^2 = p\qm p = b\qm a\qm b \qm a = b\qm 0 \qm a = p \,,\]
we have $p=p^2$; thus $p$ is an idempotent element of $J$. By the remarks above, there is a
hermitian idempotent $e\in\fu{A}$ such that $pe=e$ and $ep=p$\/; in particular, $e\in J$, since $J$ is a \emph{right} ideal in $A$ and hence in~$\fu{A}$. Then, using the tracial property of $\tau$, we have
\[ \tau(e) = \tau(pe) = \tau(ep)=\tau(p)=0\,.\]
But since $e\in J$ and $e=e^*e$ has trace zero, the `faithfulness' of $\tau$ forces $e=0$, hence forces $p=0$. Thus $b\qm a = 0 = a\qm b$ and the proof is complete.
\end{proof}

Let $A$ be a $\Cst$-algebra and $A^+$ its cone of positive elements. By a \dt{trace on $A^+$}, we mean a function $\tau: A^+\to [0,\infty]$ which is $\Real_+$-linear and satisfies $\tau(a^*a)=\tau(aa^*)$ for all $a\in A$.
We say that $\tau$ is \dt{faithful} if $\tau(x)>0$ for all $x\in A^+\setminus\{0\}$.

Given a trace on $A^+$, standard procedures from $\Cst$-algebra theory ensure that there is a $2$-sided $*$-ideal $\fm_\tau \subset A$, whose positive cone coincides with $\fm_\tau\cap A^+$, and a linear tracial functional $\fm_\tau\to \Cplx$ which coincides with $\tau$ on $\fm_\tau\cap A^+$; by abuse of notation, we will denote this functional also by~$\tau$.
Moreover, the set $\fn_\tau\defeq \{x \in A : \tau(xx^*) <\infty\}$ has the same norm-closure in $A$ as $\fm_\tau$ does.
See \cite[Proposition 6.1.2]{Dix_Cstar_en} for the details.
It is then clear that Proposition~\ref{p:qdf-cstar} has the following consequence.

\begin{cor}\label{c:qdf-via-trace}
Let $A$ be a $\Cst$-algebra and $\tau:A^+\to[0,\infty]$ a faithful trace. Then the norm closure of $\fn_\tau$ in $A$ is~q.d.f.
\end{cor}

\bigskip

\begin{rem}
A corollary of Proposition~\ref{p:qdf-cstar} is that continuous-trace $\Cst$-algebras (\cite[Ch.~4, \S5]{Dix_Cstar_en}) are~q.d.f. In this context it is natural to wonder which Type I $\Cst$-algebras are q.d.f. On the one hand, all CCR algebras are q.d.f (as remarked in \cite{YC_surjunc}), but that there are natural examples of GCR algebras which are not q.d.f, such as the Toeplitz algebra~$\cT$.
\end{rem}
\end{section}

\begin{section}{Applications to group $\Cst$-algebras}
The three most obvious $\Cst$-algebras associated to a locally compact group $G$ are the reduced group $\Cst$-algebra, the full group $\Cst$-algebra, and the group von Neumann algebra $\VN(G)$.
It is in fact well known when $\VN(G)$ is directly finite: we give the characterization for sake of completeness and for background context.
We shall then turn to examples where $\Cst_r(G)$ is q.d.f: here our main result is Theorem~\ref{t:unimodular-qdf}, which can be thought of as an extension of the main result from~\cite{Mon_dirfin}.
Lastly we shall make some comments on $\Cst(G)$.
%

\subsection*{The case of $\VN(G)$}
Being unital, a von Neumann algebra is q.d.f. if and only if it is directly finite; and this in turn happens if and only if it is a \dt{finite von Neumann algebra}, that is, one in which the identity is a finite projection.
(Briefly: if $\cM$ is directly finite, then every isometry in $\cM$ is unitary; this forces $\id[\cM]$ to be a finite projection. Conversely, if $\cM$ is finite, then standard von Neumann algebra theory tells us that $\cM$ supports a separating family of faithful finite tracial states; applying Proposition~\ref{p:qdf-cstar} we deduce that $\cM$ is directly finite.)

It is known that we can get summands of any type in the type decomposition of $\VN(G)$ by choosing appropriate locally compact groups $G$: see \cite{Suth_typeVNG} for some of the possibilities.
On the other hand, it has long been known exactly when $\VN(G)$ is a finite von Neumann algebra. To state this characterization we need the following terminology: a topological group $G$ is said to have \dt{small invariant neighbourhoods}, or to be a \dt{\SIN\ group}, if it has a neighbourhood base at the identity consisting of conjugation-invariant neighbourhoods. Any compact, abelian or discrete group has this property, for example.

\begin{thm}
Let $G$ be a locally compact group. Then $\VN(G)$ is a finite von Neumann algebra if and only if $G$ is a \SIN\ group.
\end{thm}

This result is essentially \cite[Proposition 13.10.5]{Dix_Cstar_en}. (To be precise: the result there is stated under the additional assumption that $G$ is unimodular, but inspection of the argument shows this extra condition to be unnecessary. For a less compressed account, see e.g.~Propositions 3.2 and~4.1 of \cite{Tay_typeVNG} and the surrounding remarks.)

In particular, since $\SL(2,\Real)$ is known to be non-\SIN, we deduce that $\VN(\SL(2,\Real))$ is not q.d.f. (Alternatively, we could have appealed to known structure theory for $\VN(\SL(2,\Real))$, which tells us that it has a unital subalgebra isomorphic to~$\Bdd(\ell^2)$.)

\subsection*{Results for $\Cst_r(G)$}
The simplest case to consider is when $G$ is unimodular (thus including all \SIN\ groups, but also those connected Lie groups which are semisimple or nilpotent).

We recall the definition of the \dt{Plancherel weight} on $\VN(G)$.
Let $\cF$ be the family of all compact, finite-measure neighbourhoods of the origin in $G$, and define $\tau: \VN(G)^+\to [0,\infty]$ by
\begin{equation}
\tau(a) \defeq \sup_{K\in\cF} \abs{K}^{-2} \pair{ a(\chi_K)}{\chi_K}
\end{equation}
Clearly $\tau$ is faithful, and it is finite-valued on $C_c(G)^+$.
Moreover, an easy calculation shows that if $G$ is unimodular then $\tau$ is a trace on $\VN(G)^+$. Since $C_c(G)\subseteq \fn_\tau$ and $\fm_\tau$ shares the same norm closure as $\fn_\tau$,
applying Corollary~\ref{c:qdf-via-trace} yields the following result.

\begin{thm}\label{t:unimodular-qdf}
Let $G$ be a unimodular, locally compact group. Then $\Cst_r(G)$ is q.d.f.
\end{thm}

By Theorem~\ref{t:surjunc} and the remarks which follow it, this gives us several examples of natural group representations $\theta:G\to\Bdd(X)$ where the algebra $\theta(L^1(G))$ is surjunctive.

\begin{rem}
Since semisimple Lie groups are unimodular, Theorem~\ref{t:unimodular-qdf} provides another proof of \cite[Theorem~3.4]{YC_surjunc}.
\end{rem}

\begin{rem}
It is unclear whether having a dense q.d.f.\ *-subalgebra is always sufficient for a $\Cst$-algebra to be q.d.f.
\end{rem}

\medskip
Although the Plancherel weight on $\VN(G)^+$ is tracial if and only if $G$ is unimodular, there are non-unimodular groups $G$ for which $\VN(G)$ is a semifinite von Neumann algebra (hence supports some \emph{other} faithful normal semifinite trace). This happens, for instance, if $G$ is connected and separable (see~\cite{Dix_semifin,Puk_fixDix} for the details, and also \cite[pp.~248--249]{Suth_typeVNG} for a more general perspective). 
In such situations, it is tempting to hope that this trace $\tau$ on $\VN(G)^+$ can be chosen such that $\fm_\tau\cap\Cst_r(G)$ is dense in $\Cst_r(G)$.

However, this does not work: we can find rather basic examples of connected, solvable Lie groups
whose (reduced) group $\Cst$-algebras are not q.d.f.
This seems to be well-known to specialists: a particularly simple and explicit example is discussed by Rosenberg in~\cite{Ros_Pac76}. Since the emphasis in that paper is rather different from the subject of this article, we shall for sake of convenience outline the salient facts.
Let $G$ be the ``complex $ax+b$ group''
\[ G = \left\{ \left(\begin{matrix} a & b \\ 0 & 1 \end{matrix}\right) \;:\; a\in\Cplx^*, b\in\Cplx\right\}. \]
Up to unitary equivalence, $G$~has precisely one infinite-dimensional unitary representation, which we denote by $\sigma$; this representation is faithful in the sense that $\sigma: \Cst(G)\to\Bdd(H_\sigma)$ is injective. (See the remarks on \cite[p.~177]{Ros_Pac76}.)

We then have the following result, whose proof can be found in that of 
\cite[Proposition~1]{Ros_Pac76}.

\begin{prop}\label{p:countereg}
There exists $\varphi\in L^1(G)$ such that $I+\sigma(\varphi)$ is injective with closed, proper range. (In fact it has a one-dimensional cokernel.)
\end{prop}

\begin{cor}
$\Cst_r(G)$ is not q.d.f.
\end{cor}

\begin{proof}[Proof of the corollary]
We start by noting that since $G$ is amenable, we can work with the full $\Cst$--algebra instead of the reduced one.

By Proposition~\ref{p:countereg}, $-1$ lies in the residual spectrum of $\sigma(\varphi)$.
Since $\sigma$ is faithful, if $\Cst(G)$ were q.d.f. then $\sigma(\Cst(G))$ would be surjunctive, by Theorem~\ref{t:surjunc}. But this contradicts our initial observation that $\sigma(\varphi)$ has non-empty residual spectrum.
\end{proof}


Rosenberg's calculations actually went much further, and determined (up to isomorphism) the reduced group \Cst-algebras of certain relatives of the $ax+b$ group. The methods, which use a mix of commutative Fourier analysis with BDF techniques, were extended by Wang in \cite{Wang_pitman199} to the class of groups $G(p,q,\al)$: here $p$, $q$ are positive integers, and $G(p,q,\al)$ is the semidirect product arising from an action of $\Real$ on $\Real^{p+q}$ via
\[ t \mapsto \operatorname{diag}(e^{\al_1 t},\dots, e^{\al_p t},
	e^{-\al_{p+1} t}, \dots, e^{-\al_{p+q} t}) \]
where $\al_1,\dots,\al_{p+q}>0$.
{\it En route}\/ to determining $\Cst(G(p,q,\al))$, the following result is proved.

\begin{thm}[{\cite[Corollary 3.2]{Wang_pitman199}}]
There is an embedding of $\Cst(G(p,q,\al))$ as a subalgebra of $C(S^{p-1}\times S^{q-1}, \Cpct(L^2(\Real)))$.
\end{thm}

\begin{cor}
$\Cst(G(p,q,\al))$ is q.d.f.
\end{cor}

\begin{rem}
An alternative proof that $\Cst(G(p.q,\al))$ is q.d.f.\ goes as follows. First, note that the isomorphism class of $\Cst(G(p,q,\al))$ depends only on $p$ and $q$ and not on the lengths of the `roots' $\al_1,\dots,\al_{p+q}$. (See \cite[pp.~12--13]{Wang_pitman199} and \cite[pp.~190--191]{Ros_Pac76}.)
Secondly, note that if we choose the $\al_i$ so that $\al_1+\dots+\al_p=\al_{p+1}+\dots+\al_{p+q}$, then for each $x$ in the Lie algebra $\fg$ of $G(p,q,\al)$, the operator $\operatorname{ad}_x:\fg\to\fg$ has trace zero; and this is known (see the remarks on \cite[p.~190]{Ros_Pac76}) to imply that the group is unimodular. Applying Theorem~\ref{t:unimodular-qdf}, we deduce that $\Cst(G(p,q,\al))$ is~q.d.f for this -- and hence any -- choice of the $\alpha_i$.
%
\end{rem}

\subsection*{The full group $\Cst$-algebra}
What can be said for the \emph{full} group $\Cst$-algebra of $G$? If $G$ is amenable then $\Cst(G)=\Cst_r(G)$ and we can apply the results of the previous section. If $G$ is non-amenable, then the situation is unclear even for discrete groups.

There exist nonamenable discrete groups $\Gm$ for which $\Cst(\Gm)$ has a faithful tracial state, and hence is q.d.f. by Proposition~\ref{p:qdf-cstar} -- for instance, this happens if $G$ is residually finite. On the other hand, in \cite{BekLou_trace} Bekka and Louvet give examples of discrete groups~$\Gm$ for which $\Cst(\Gm)$ has no faithful trace; but I do not know if such examples fail to be q.d.f.

For certain classes of connected groups, we can do better. It was observed in \cite{YC_surjunc} that if $G$ is a CCR group (for instance, a connected Lie group which is either semisimple, nilpotent, or real algebraic) then $\Cst(G)$ has directly finite unitization, and hence in the language of this article is~q.d.f.

\end{section}

\begin{section}{Final remarks}
Many questions remain. We shall close with three in particular.
\begin{enumerate}
\item Does there exist a discrete group $\Gm$ (necessarily non-amenable) such that $\Cst(\Gm)$ is not q.d.f?
\item Let $G$ be a connected Lie group with Lie algebra $\fg$. Are there ``reasonable' characterizations in terms of $\fg$ for when $\Cst_r(G)$ is q.d.f? What if we restrict attention to solvable Lie groups?
\item If $G$ is not \SIN, then as remarked above its group von Neumann algebra cannot be q.d.f. What about its measure algebra~$M(G)$?
\end{enumerate}

\end{section}

\section*{Acknowledgements}
My thanks go to C.~Zwarich for pointing out to me the characterization of when the group von Neumann algebra is finite, 
R,~Archbold for several comments and corrections,
and M.~Daws for useful conversations leading to a more streamlined presentation of Theorem~\ref{t:unimodular-qdf}.

\def\arX#1{{\tt arXiv} #1}

\begin{thebibliography}{10}

\bibitem{BekLou_trace}
{\sc M.~B. Bekka} and {\sc N.~Louvet}, {\em Some properties of {$C^*$}-algebras
  associated to discrete linear groups}, in {$C^*$}-algebras ({M}\"unster,
  1999), Springer, Berlin, 2000, pp.~1--22.

\bibitem{YC_surjunc}
{\sc Y.~Choi}, {\em Group representations with empty residual spectrum}, Int.
  Eq. Op. Th., 67 (2010), pp.~95--107.

\bibitem{Dix_BSMF53}
{\sc J.~Dixmier}, {\em Formes lin\'eaires sur un anneau d'op\'erateurs}, Bull.
  Soc. Math. France, 81 (1953), pp.~9--39.

\bibitem{Dix_semifin}
{\sc J.~Dixmier}, {\em Sur la repr\'esentation r\'eguli\`ere d'un groupe
  localement compact connexe}, Ann. Sci. \'Ecole Norm. Sup. (4), 2 (1969),
  pp.~423--436.

\bibitem{Dix_Cstar_en}
\leavevmode\vrule height 2pt depth -1.6pt width 23pt, {\em {$C\sp*$}-algebras},
  North-Holland Publishing Co., Amsterdam, 1977.
\newblock Translated from the French by Francis Jellett, North-Holland
  Mathematical Library, Vol. 15.

\bibitem{Eym_BSMF64}
{\sc P.~Eymard}, {\em L'alg\`ebre de {F}ourier d'un groupe localement compact},
  Bull. Soc. Math. France, 92 (1964), pp.~181--236.

\bibitem{Mon_dirfin}
{\sc M.~S. Montgomery}, {\em Left and right inverses in group algebras}, Bull.
  Amer. Math. Soc., 75 (1969), pp.~539--540.

\bibitem{Munn_DF1}
{\sc W.~D. Munn}, {\em Direct finiteness of certain monoid algebras.~{I}},
  Proc. Edinburgh Math. Soc. (2), 39 (1996), pp.~365--369.

\bibitem{Puk_fixDix}
{\sc L.~Pukanszky}, {\em Action of algebraic groups of automorphisms on the
  dual of a class of type {${\rm I}$} groups}, Ann. Sci. \'Ecole Norm. Sup.
  (4), 5 (1972), pp.~379--395.

\bibitem{Ros_Pac76}
{\sc J.~Rosenberg}, {\em The {$C\sp*$}-algebras of some real and {$p$}-adic
  solvable groups}, Pacific J. Math., 65 (1976), pp.~175--192.

\bibitem{Suth_typeVNG}
{\sc C.~E. Sutherland}, {\em Type analysis of the regular representation of a
  nonunimodular group}, Pacific J. Math., 79 (1978), pp.~225--250.

\bibitem{Tay_typeVNG}
{\sc K.~F. Taylor}, {\em The type structure of the regular representation of a
  locally compact group}, Math. Ann., 222 (1976), pp.~211--224.

\bibitem{Wang_pitman199}
{\sc X.~Wang}, {\em The {$C^*$}-algebras of a class of solvable {L}ie groups},
  vol.~199 of Pitman Research Notes in Mathematics Series, Longman Scientific
  \& Technical, Harlow, 1989.

\end{thebibliography}

\vfill

\noindent%
\begin{tabular}{l}
Y. Choi\\
D\'epartement de math\'ematiques
et de statistique,\\
Pavillon Alexandre-Vachon\\
Universit\'e Laval\\
Qu\'ebec, QC \\
Canada, G1V 0A6 \\
{\bf Email: \tt y.choi.97@cantab.net}
\end{tabular}

\end{document}